\documentclass[a4paper, 12pt]{amsart}
\usepackage{amssymb,amsmath,txfonts,mathrsfs,titletoc,appendix}
\usepackage{graphicx}
\usepackage{latexsym}
\usepackage[alphabetic]{amsrefs}
\usepackage{geometry}
\usepackage{fancyhdr}
\usepackage{xcolor}
\usepackage{comment}
\usepackage{subcaption}
\usepackage{pifont} %%% added 20230429

\usepackage{exscale}        %added 20230507
\usepackage{relsize}       %added 20230507

\usepackage{tensor}   %added 20230601

\usepackage{float} 
\usepackage[new]{old-arrows}
\usepackage{longtable}
\usepackage{booktabs}
\usepackage{array}
\usepackage{multirow}
\usepackage{supertabular}
\usepackage{amssymb}
\usepackage{amsmath}
\usepackage{mathrsfs}
\usepackage{titletoc}
\usepackage{latexsym}
\usepackage[alphabetic]{amsrefs}

 \newtheorem{thm}{Theorem}[section]
 \newtheorem{cor}[thm]{Corollary}
 \newtheorem{lem}[thm]{Lemma}

 \newtheorem{rem}[thm]{Remark}
 \numberwithin{equation}{section}
\newtheorem{lem*}{Lemma}
\newtheorem{cor*}{Corollary}

\newenvironment{pf}{\medskip \noindent
{\bf Proof of Theorem \ref{mainthm1}.}}{\hfill \rule{.5em}{1em}
\\}

\newenvironment{rmk}{\mbox{ }\\{\bf  Remark}\mbox{ }}{
\hfill $\Box$\mbox{}\bigskip}

\def\R{\Bbb{R}}
\def\C{\Bbb{C}}

\begin{document}
\title{%{\bf\tt FURTHER Modified 20230223 Draft of}\\
Remarks on ``Spiral Minimal Products"}
%\thanks{H. Li was partially supported by NSFC (Grant No. 11831005). 
%\\\text{\, \, }
%Y. Zhang was sponsored in part by NSFC (Grant Nos. 12022109 and 11971352).}
%\dedicatory{Dedicated to Prof. F. Reese Harvey and Prof. H. Blaine Lawson, Jr.  with admiration}
\author{Haizhong Li}
\address{Department of Mathematical Sciences, Tsinghua University, Beijing 100084, P. R. China}
\email{lihz@mail.tsinghua.edu.cn}
\author{Yongsheng Zhang$^*$}
\address{Academy for Multidisciplinary Studies, Capital Normal University, Beijing 100048, P. R. China}
\email{yongsheng.chang@gmail.com}
\date{\today}
%\date{\today}
%\thanks{Partially sponsored by the Fundamental Research Funds for the Central Universities, the SRF for ROCS, SEM
%and 
%the NSF under Grant No. 0932078 000,
%while the author was in residence at the MSRI during the 2013 Fall.}

%\keywords{$\mathscr C$-totally real minimal submanifold, 
%special Legendrian submanifold, totally real submanifold, minimal Lagrangian submanifold,
%horizontal lifting, special Lagrangian calibration} 
%%\subjclass[2010]{~53C38,~49Q15, ~28A75}
\begin{abstract}
This note aims to give a better understanding and some remarks about recent preprint ``Spiral Minimal Products".
In particular, 1. it should be pointed out that a generalized Delaunay construction among minimal Lagrangians of complex projective spaces has been set up.
This is a general structural result working for immersion and current situations.
2. uncountably many new regular (or irregular) special Lagrangian cones with finite density and ``regular" (or irregular) special Lagrangian cones with infinite density in complex Euclidean spaces can be found.
\end{abstract}
\maketitle
%\titlecontents{section}[0em]{}{\hspace{.5em}}{}{\titlerule*[1pc]{.}\contentspage}
%\titlecontents{subsection}[1.5em]{}{\hspace{.5em}}{}{\titlerule*[1pc]{.}\contentspage}
%{\setcounter{tocdepth}{2} \small \tableofcontents}
\section{Review}\label{P} 

       In \cite{LZ}, we consider spiral product by an immersed curve $\gamma=(\gamma_1,\gamma_2)$ in the unit Euclidean sphere $\mathbb S^3\subset \mathbb C\oplus \mathbb C$. 
        Given embedded $M_1^{k_1}\subset \mathbb S^{2n_1+1}\subset \mathbb C^{n_1+1}$ and $M_2^{k_2}\subset \mathbb S^{2n_2+1}\subset \mathbb C^{n_2+1}$, 
          their spiral product $G_\gamma$ for $\gamma$
          is
            \begin{eqnarray}
 \ \ \ \ \ \  \ \ \  G_\gamma: \mathbb R\times M_1\times M_2 \longrightarrow \mathbb S^{n_1+n_2+1}\ \ by \ \ 
            %by
                               \big(t,x,y\big)\longmapsto \Big(\gamma_1(t)f_1(x),\, \gamma_2(t)f_2(y)\Big) .
                               \label{expression}
   \end{eqnarray}

   For simpler computations, we focus on the situation that
   both inputs $M_1$ and $M_2$ are $C$-totally real,
   \footnote{We shall use $C$-totally real, $C_1$ and $C_2$ to replace $\mathscr C$-totally real, $C$ and $\tilde C$ in \cite{LZ}.}\label{ft1}
   namely, $Jx\perp T_xM_1$ and $Jy\perp T_yM_2$ for $\forall x\in M_1, y\in M_2$.
   Here $J$ means the standard complex structure of $\mathbb C^{n_1+1}$ and $\mathbb C^{n_2+1}$ respectively. 
   
   When both inputs are $C$-totally real minimal submanifolds, we want to derive spiral minimal products of them by some $\gamma$.
   With respect to preferred tangential and normal orthogonal bases, 
    the question of minimal surface PDE system transforms to solving a pair of ODEs
     
\begin{equation}\label{1}
 %
    %                       &=&
     %                        -2 a'b s'_1 s'_2
   %%     -abs''_1s'_2 %%%%added 2023.2.22
     %                        + 2 a b' s'_1 s'_2
    %    +abs'_1s''_2%%%%added 2023.2.22
                        %    =
                               0=     -2 s'_1 s'_2 b^2 \Big(\dfrac{a}{b}\Big)'
                                    -ab{s'_2}^2 \Big(\frac{s'_1}{s'_2}\Big)'
                                   % =0
\end{equation}
and  \ \  $0=\big[(a')^2+(b')^2+\Theta\big]\big(-k_1\frac{b}{a}+k_2\frac{a}{b}\big)
+
      %      
          %  {\ \ \ \ \ \ \ }        &+&
                         %              \frac{1}{V} 
                                       \Bigg\{
                                          \Big[
                                          a''-a(s'_1)^2
                                          \Big]
                                          b
                                                                                 -
                                                       \Big[
                                                       b''-b(s'_2)^2
                                                       \Big]
                                                       a
                                                                                 \bigg\}
%\label{2}
$
\begin{eqnarray}
      &&          \ \     \ \ \ \ \    \ \ \ \ \ \ \ \     \ \ \ \    \ \ \ \ \ \ \ \ \,\,\,\,\,\,\,        \ \ \  \ \ \ \ \ %\ \ \ \ \ \ \ \ \                            
        -
                                       \frac{\mathcal V}{\Theta}
                                           \bigg\{
                                              \Big(2a's'_1+as''_1\Big)as'_1+\Big(2b's'_2+bs''_2\Big)bs'_2
                                          \bigg \}.
                   \label{2}                
\end{eqnarray}
             Here (for $G_\gamma$ to be an immersion) $a=|\gamma_1|\neq 0$, $b=|\gamma_2|\neq 0$, $s_1=arg \gamma_1$,  $s_2=arg \gamma_2$, 
             $\mathcal V=a'b-ab'$,      $\Theta=(as'_1)^2+(bs'_2)^2$ and we only consider things in local  not bothering the case that $s_1'\equiv s_2'\equiv 0$.
      It seems hopeless to solve them at first glance. 
                  However,  by using the arc parameter for curve $(a, b)$ with $a(s)=\cos s$ and $b(s)=\sin s$,
                                          %          \footnote
                                         %           {
                                          %          Although it seems that we take a peculiar choice of parametrization in analysis, 
                                          %          actually we do not lose any information (of solution $\gamma$ in local) from the geometric viewpoint.
                                          %          In the current situation, %by Allard \cite{Allard} or Morrey \cite{M1, M2},
                                           %         curve $(a(t), b(t))$ has nonvanishing velocity in local.
                                          %           Any immersed (connected piece of) curve has arc parameters
                                           %          and,
                                          %           compared with the canonical anti-clockwise arc parameter $s$ of $\mathbb S^1$
                                          %           measuring oriented 
                                                     %arc 
                                           %          length from 
                                                     %the 
                                          %           starting point $(1,0)$,
                                          %           all arc parameters have to be 
                                                     %of format 
                                           %          $\pm s$ plus some constant.
                                                    % which can be gained from requirements 
                                                    % $a^2+b^2=1$ and $(\dot a)^2+(\dot b)^2=1$.
                                           %          }
the ODEs can be fortunately solved by

  \begin{equation}\label{dotsss}
                 \begin{pmatrix}
       \dot s_1\\
      % \\
     \dot s_2
       \end{pmatrix}
       = 
                %%%%
                               \pm               
                             \sqrt{\frac{1}{ 
                                   {C_2  
                 \Big(\cos s\Big)^{2k_1+2}
                  \Big( \sin s\Big)^{2k_2+2}
                  \,
              %      \big|
              %            \cos s \sin s
              %             \big|
                   -
                                    %  \sin^2 s
                 %  \sqrt
                   {
                 %  \big|
                      1-\big(C_1^2-1\big)\cos^2 s
                  %    \Big|
                   }
            }
            }
 }
           \,     %%%%
       \begin{pmatrix}
       \tan s\\
     %  \\
     C_1\cot s
       \end{pmatrix}.
                \end{equation}
         Here 
         $ C_1=  \frac{b^2\dot s_2}{a^2\dot s_1}$
         serves as ratio of angular momenta of complex components of $\gamma$ 
         and
         \begin{equation}\label{C2}
         C_2>\min_{s\in (0,\frac{\pi}{2})}
                     \frac{ 1+(C_1^2-1)\cos^2 s}{
                  \Big(\cos s\Big)^{2k_1+2}
                  \Big( \sin s\Big)^{2k_2+2}
                  }
                   \end{equation}
for \eqref{dotsss} making sense in some nonempty set.
Not hard to see that solutions of \eqref{C2} can form a connected open interval in $(0,\frac{\pi}{2})$.
       Denote it  by $\Omega^0_{C_1,C_2}$ and the solution curve over it by $\gamma^0$.

    It should be pointed out that
    the starting $argument$s are not essential since we can use $(e^{i\theta_1}, e^{i\theta_2})$ where $\theta_1, \theta_2\in \mathbb R$ to move $\gamma^0$.
    Note that this commutes with the generating action $\gamma$ on the ambient Euclidean space.
    Now we can assemble $+$ and $-$ parts of \eqref{dotsss} together alternately to get a ``complete"  solution curve 
    $\gamma:
             \R=\cdots \bigcup \Omega^0_{C_1,C_2}\bigcup \Omega^1_{C_1,C_2}\cdots
             \longrightarrow 
             \mathbb S^3
             $.
Based on Harvey-Lawson's extension result for minimal submanifolds
               with $C^1$ joints
               we know that $\gamma$ is analytic.
                  \begin{figure}[h]
	%\centering
	\begin{subfigure}[t]{0.3\textwidth}
		\centering
		\includegraphics[scale=0.45]{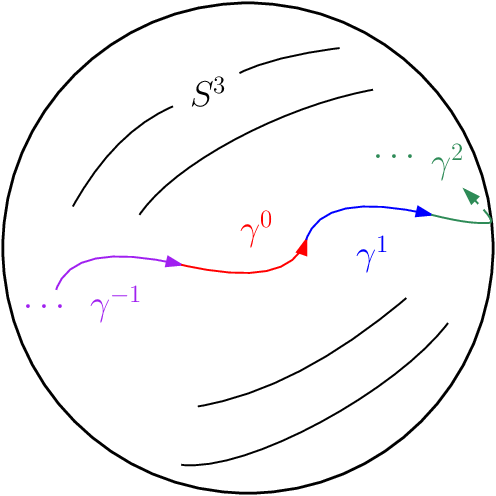}
                              \captionsetup{font={tiny}} 
                             \caption{Complete $\gamma$ in $\mathbb S^3$}
                               \label{fig:1a}
	\end{subfigure}
	\quad\quad \quad\quad 
	\begin{subfigure}[t]{0.53\textwidth}
		\centering
		\includegraphics[scale=0.58]{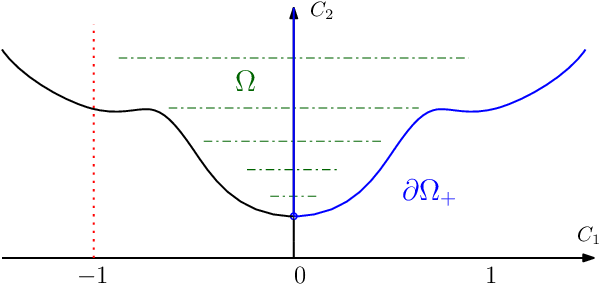}
                              \captionsetup{font={tiny}} 
                             \caption{Domain of allowed $(C_1,C_2)$}
                               \label{fig:1a}
	\end{subfigure}
	 \caption{Generating ``complete" solution curves}
%	\ \ \ \ \ \ \ \ \ \ \ \ \ \ \ \ \ \ \ \ \
	\end{figure}
So, each point $(C_1, C_2)\in \Omega$ in Figure (B) determines a solution curve.
As a result, we get uncountably many spiral minimal products based on $C$-totally real minimal $M_1$ and $M_2$
and one can apply the algorithm repeatedly for multiple inputs.
Note that the machinery works perfectly in the category of immersions.
 Since every minimal submanifold in sphere  becomes $C$-totally real minimal in some higher dimensional sphere,
the spiral minimal product construction here implies that the moduli spaces of $C$-totally real minimals can be quite big.

{\ }

\section{Global horizontal lifting}\label{P} 
              The most interesting case regarding spiral minimal products may be
              the situation with $C_1=-1$.
              The first observation is that
              every spiral product $G_\gamma$ for $\gamma$ (not necessarily a solution curve) with $C_1=-1$ 
              is $C$-totally real if both inputs $M_1$ and $M_2$ are $C$-totally real.
              Thus, in conjunction with Hopf projection $\pi$, one can further get immersed submanifolds in complex projective space.
              %Now $G_\gamma$ is its horizontal lift.

%{\ }

              It is worth noting that, if an immersed submanfold $M'$ in $\mathbb CP^n$ is totally real,
              i.e., the complex structure maps its tangential space into its normal space,
                then, around every point, $M'$ has a local horizontal lift (that means exactly a $C$-totally real lift).
                Another well-known fact is that $M'$ is minimal if and only if its (local) horizontal lift is minimal.
                
                When a totally real $M'$ attains the largest possible dimension $n$, we say  it is  a Lagrangian submanifold in $\mathbb CP^n$.
                Similarly,  a $C$-totally real submanifold of $\mathbb S^{2n+1}$ which reaches largest possible dimension $n$ is called Legendrian,
                and the counterpart of dimension $n+1$ in $\mathbb C^{n+1}$  again called Lagrangian.

A key lemma which establishes global correspondence is the following.

%{\ }
                    
                     \begin{lem}[\cite{LZ}]\label{key}
               Given an $n$-dimensional connected embedded minimal Lagrangian submanifold $M'\subset \mathbb CP^n$.
               Then it has a connected embedded horizontal lift $M^n\subset \mathbb S^{2n+1}\subset \mathbb C^{n+1}$
               %satisfying 
               such
               that the Hopf projection $\pi: M\longrightarrow M'$ gives an $\ell:1$ covering map %of degree $\ell$
               where $\ell$ is an integer factor of $2(n+1)$.
               Moreover, as a set, $e^{\frac{2\pi i}{\ell}}\cdot M=M$.
               \end{lem}
               
               \begin{figure}[h]
	\centering
	\begin{subfigure}[t]{0.45\textwidth}
		\centering
		\includegraphics[scale=0.7]{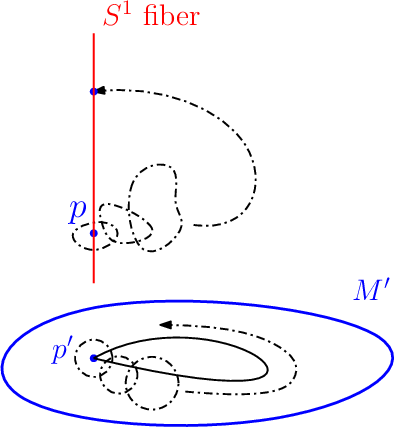}
                              \captionsetup{font={scriptsize}} 
                               \caption{Patching along loops}
                               \label{fig:2a}
	\end{subfigure}
	%\quad\quad \quad\quad 
	\begin{subfigure}[t]{0.48\textwidth}
		\centering
	\ \ \ \ \ \ \  \includegraphics[scale=0.68]{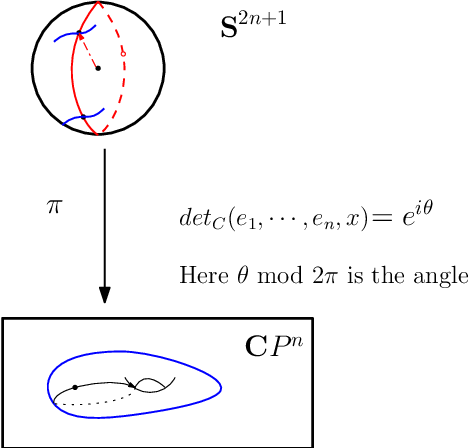}
                     \captionsetup{font={scriptsize}} 
                               \caption{Legendrian/Lagrangian angle}\label{fig:2b}
	\end{subfigure}
	%\caption{Pendulum phenomenon for singly spiral minimal products} 
 \caption{Lifts subject to Legendrian/Lagrangian angle}
\end{figure}
 
             %  $$
            %   \includegraphics[scale=0.65]{Hlifts.eps}
             %  $$
               
               The method we use to prove Lemma \ref{key} is to patch local horizontal lifts along curves and look at the Legendrian angle
    \footnote{
                      Let $\{e_1,\cdots, e_n\}$ be an oriented local orthonormal frame of the tangent space of the Legendrian submanifold and $x$ the position vector (see Figure \ref{fig:2b}). 
                           Then, with respect to the standard complex basis of $\mathbb C^{n+1}$, 
                              the $(n+1)\times (n+1)$ matrix $(e_1,\cdots, e_n, x)$ has a $\mathbb C$-determinant of norm one.
                                   \label{ft2}}
                of a minimal Legendrian submanifold,
               which equivalently is the Lagrangian angle of the cone over the minimal Legendrian (the cone is then special Lagrangian calibrated by the canonical calibration).
               It can be shown that $\ell$ has to be an integer factor of $n+1$ if $M'$ is orientable and otherwise that of $2(n+1)$.
               \footnote{
                         E.g. minimal Lagrangian $\R P^2\subset \mathbb CP^2$ when $n=2$.
                          \label{ft3}
                          }
              
               In fact our method can be refined to deal with immersed situations of closed submanifolds.

  \begin{cor}\label{corimm}
               Given an $n$-dimensional connected immersed closed minimal Lagrangian submanifold $M' \looparrowright \mathbb CP^n$ (as a map).
               Then it has a global horizontal lift $M \looparrowright \mathbb S^{2n+1}\subset \mathbb C^{n+1}$ 
               given by an immersed of a connected closed manifold $M$.
               Moreover,  the immersed $M$ can only have self-intersection  of codimension $\geq 2$ (in $M$).
               \end{cor}
               
               \begin{proof}
               By the compactness, it follows that at point $p'\in \mathbb CP^n$ there are at most finitely many local embedded pieces of $M'$ passing through it.
               Then by the finiteness and the arguments to prove Lemma \ref{key} we know that again we can get  
                       an immersion (into $\mathbb S^{2n+1}$) of some connected closed manifold $M$ 
                                  possibly with higher codimension self-intersection
                                     as a global horizontal  lift of $M'$.
               Codimension-one self-intersection of the immersed $M$ cannot occur due to
                the fact that cone over the immersed $M$ is special Lagrangian 
                and 
                that one can apply a calibration argument or the Almgren big regularity theorem. 
               \end{proof}
               
               \begin{rem}\label{rk4imm}
               Any codimension-one self-intersection of $M'$ downstairs must be dissolved by assembling local horizontal lifts.
               Part of self-intersection of higher codimension may possibly survive in the global horizontal lift $M \looparrowright   \mathbb S^{2n+1}$.
               Following the patching procedure, around every point of the abstract $M'$ (before immersion into $\mathbb CP^n$), corresponding pieces of abstract $M$ form an $\ell$-fold cover.
               By the connectedness of $M'$, the fold number $\ell$ is constant everywhere.
           Not hard to see that if $p\in\mathbb S^{2n+1}$ is a self-intersection point of the immersed $M$ 
           then so is $e^{\frac{2\pi i}{\ell}}p$.
               \end{rem}
               
               Even further we want to extend Corollary \ref{corimm} to include
                     %  ($d$-closed)  
                       stationary Lagrangian 
                       %(oriented) 
                       integral currents $T'$ in $\mathbb CP^n$ (say with multiplicity one) with compact support,  connected regular part and  no boundary.
                         Here by stationary Lagrangian 
                                 we mean that the integral current
                                     is stationary and its tangent cone is Lagrangian a.e.
                         One can apply the patching argument for  Lemma \ref{key} based at any regular point.
                         Then by taking closure one can get 
                         the following.
                     
                       \begin{cor}\label{corcurrent}
               Given stationary Lagrangian current $T'$ as mentioned above in  $ \mathbb CP^n$.
               Then it has a special Legendrian current $T$ as global horizontal lifting in  $\mathbb S^{2n+1}\subset \mathbb C^{n+1}$.
              % Moreover,  the immersed $M$ can only have self-intersection  of codimension $\geq 2$ (in $M$).
               \end{cor}

                   \begin{figure}[h]
	\centering
               \includegraphics[scale=0.75]{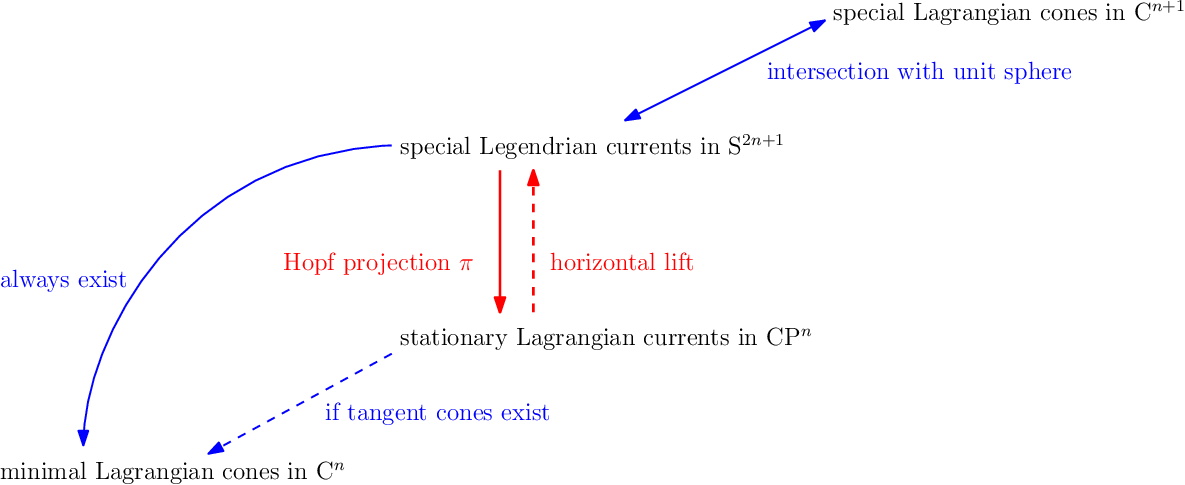}
               \caption{\small Global picture of the framework}
          \end{figure}

               A minimal Lagrangian cone is a minimal cone which is Lagrangian a.e.
               One may encounter local example like 
                   span$_\R\{
                               %  \overrightarrow 
                                 1, 
                                 % \overrightarrow
                                   j\}$ 
                                $\bigcup$ 
                                  span$_\R\{
                                             % \overrightarrow 
                                              i, 
                                                  %  \overrightarrow 
                                                    k\}$ 
                                         in quaternion $\mathbb H$.
                                         Away from the origin, the cone is Lagrangian (different pieces may have different Lagrangian angles).
               Although its local regular part is disconnected, the patching horizontal lifts can still run through all regular part due to the connectedness assumption.
                         
                         Another  possibility which may cause disconnectedness of the regular part of $T'$ is the codimension-one singularities.
                         According to the \cite{NV}, 
                         the singular set of $T'$ has a stratiﬁcation structure 
                         with top level $\mathcal S^{n-1}-\mathcal S^{n-2}$, if not empty, $(n-1)$-rectifiable.
                       %  Recall that dim$\mathcal S^{j}\leq j$.
                        If at some $p'$ in $\mathcal S^{n-1}-\mathcal S^{n-2}$ 
                        tangent cone of $T'$ exists and $T_{p'} \mathcal S^{n-1}=\R^{n-1}$, 
                        then
                         it must have the open-book structure along the spine $T_{p'} \mathcal S^{n-1}$  according to  Allard's boundary regularity paper \cite{A}.
                         Due to the stationary assumption, either the cone is several $n$-planes through the spine 
                               or otherwise a different kind of collection of half $n$-planes balanced along the spine.
                         If the former occurs and $\mathcal S^{n-1}-\mathcal S^{n-2}$ were  $C^{1, \alpha}$,
                         then  locally these codimension-one singularities arise from self-intersections as already seen in the immersed situations
                         and moreover the patching procedure can pass through $p'\in\mathcal S^{n-1}-\mathcal S^{n-2}$.
                         However the latter situation, if existed  in local (see Figure \ref{fig:a5}),  forms an obstruction for a horizontal lift.
                         So it seems that the connectedness assumption on the regular part of $T'$ is necessary for the procedure.

                                   \begin{figure}[h]
		 \includegraphics[scale=0.85]{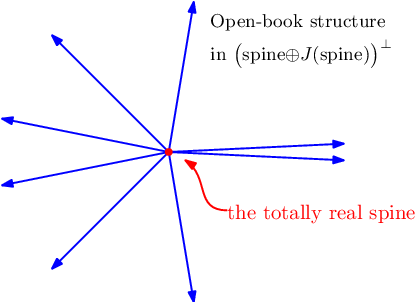}
                             \caption{A minimal Lagrangian $n$-cone with an $(n-1)$-spine in $\C^n$} 
                          \label{fig:a5}
                             \end{figure}

                         One more remark about the difference between current and immersed cases is
                         this.
                         Assume $\mathcal S^{n-1}-\mathcal S^{n-2}=\emptyset$ and $p'\in \mathcal S^{n-2}\neq \emptyset$.
                         Suppose that $w\leq n-2$ is the largest integer for $p'\in\mathcal S^{w}$.
                         
                         1. If  supp$(T')$ $-\mathcal S^{w}$ is connected inside $B_\epsilon(p')$ for all $0<\epsilon<$ some $\epsilon_0$,
                         then
                               along any small loop (staring from and ending at $q'$) avoiding $\mathcal S^{w}$ the patching horizontal lifts must coincide
                               as the lifted curve cannot have enough length to connect  different $q$ and $\tilde q$ upstairs 
                               (over the fiber for $q'$)
                               with the allowed discrete distances in Lemma \ref{key}.
                               As a result, the local singularity structure will survive in the global horizontal lift after taking the closure of the global horizontal lift for the connected regular part
                               (in proving Corollary \ref{corcurrent}). 
                               
                               2. If in any sufficiently small scale supp$(T')$ $-\mathcal S^{w}$ is not connected,
                               then in the closure of horizontal lift
                               the singularity stratiﬁcation may be decomposed and reassembled accordingly.
                               See Remark \ref{rk4imm}.
                         
                                                  %So the assumption can be weakened a bit for being able to patch local horizontal lifts for a global one.

              {\ }

             \section{Delaunay construction for minimal Lagrangians in complex projective spaces}\label{P3}

            Now the framework of \cite{LZ} can be broadened
            and
            generalized Delaunay construction for minimal/stationary Lagrangians in complex projective spaces are the followings.
                   
                   \begin{thm}[Delaunay construction 1]\label{DC1}
                   Let $M_1' \looparrowright \mathbb CP^{n_1}$
                   and
                          $M_2'  \looparrowright \mathbb CP^{n_2}$
                          be two connected immersed closed minimal Lagrangians.
                          Then, based on them, uncountably many immersed minimal Lagrangians can be constructed  in $\mathbb CP^{n_1+n_2+1}$. 
                   \end{thm}
                   
                   \begin{proof}
                   By Corollary \ref{corimm}, we have immersed closed submanifolds $M_1\looparrowright \mathbb S^{2n_1+1}$ and $M_2 \looparrowright \mathbb S^{2n_2+1}$ 
                   as horizontal liftings for $M_1'$ and $M_2'$.
                   Using any solution curve $\gamma$ with $C_1=-1$, 
                   a minimal Legendrian immersion $G_\gamma$ can be gained.
                   Furthermore, $\pi\circ G_\gamma$ gives a minimal Lagrangian immersion from $\R\times M_1\times M_2$ into $\mathbb CP^{n_1+n_2+1}$.
                   \end{proof}
                   
                   Note that, in general $\pi\circ G_\gamma$ may not induce a de Rham current, let alone an integral current.
                   One problem is that  the image of $\pi\circ G_\gamma$ may not be locally Hausdorff $(n_1+n_2+1)$-measurable.
                   The local behavior could be similar to $\bigcup_{y\in \mathbb Q}\{(x, y): x\in \R\}$ in $\R^2$.
                   The critical quantities to control the behavior of $G_\gamma$ are
                   
                              \begin{equation}\label{J}
                              J_{1}(C_2)=                                  \mathlarger{\int}_{\Omega^0_{-1,C_2}}  \dfrac{\tan s }{\sqrt{C_2\Delta^2-1}} 
                                         \, ds 
   %                                          \end{equation}
              \text{     \ \        \ \ \                   and\ \ \ \  \ \ }
     %                                         \begin{equation}\label{J2}
                               J_{2}(C_2)=                                   \mathlarger{\int}_{\Omega^0_{-1,C_2}}  \dfrac{\cot s}{\sqrt{C_2\Delta^2-1}} \, ds .
                                             \end{equation}
                   where $\Delta=(\cos s)^{n_1+1} (\sin s)^{n_2+1}$.
                   They measure the argument gage sizes of $\gamma_1$ and $\gamma_2$ (in opposite directions due to the sign of $C_1=-1$) when running through $\gamma^0$ over $\Omega^0_{-1,C_2}$.
                   Moreover, it has been proved in \cite{LZ} that 
                    \begin{equation}\label{2Delta}
                        (n_1+1) J_{1}(C_2) =(n_2+1) J_{2} (C_2).
                                            \end{equation}
                   So solution curve $\gamma$ with $C_1=-1$ and $a,b$ non-constant
                   \footnote{When $a, b$ are constant, it follows by \eqref{2Delta} that the argument slope $c=-\frac{n_1+1}{n_2+1}$ in \cite{LZ} corresponds to $C_1=-1$.
                   So the solution curve $\gamma$ now is an embedded closed curve.}
                    factors through simple closed curve if and only if $J_{1}(C_2)\in \pi \mathbb Q$.
                   This is exactly a necessary and sufficient condition for  
                      the image of $\pi\circ G_\gamma$
                       to be Hausdorff $(n_1+n_2+1)$-measurable.
                       Another issue is about the orientability. If the regular part of  the image of $\pi\circ G_\gamma$ is orientable,
                       then it  induces a stationary Lagrangian integral current with multiplicity one in $\mathbb CP^{n_1+n_2+1}$;
                       otherwise a stationary Lagrangian integral current mod 2 (see footnote \ref{ft3}).

                   Since $\mathbb C^{n_1+n_2+2}=\mathbb C^{n_1+1}\oplus \mathbb C^{n_2+1}$, with the obvious choice of homogeneous coordinates
(by slightly abusing symbols) the minimal Lagrangian immersion in Theorem \ref{DC1}, up to congruency, is  
$$
\left[
\gamma_1\cdot M_1,\, \ \gamma_2\cdot M_2
\right]
%=
%\left[
% M_1,\, \  {\gamma_2}\large/{\gamma_1}\cdot M_2
%\right]
%=
%\left[
%  {\gamma_1}\large/{\gamma_2}  \cdot M_1,\, \  M_2
%\right]
$$
in $\mathbb CP^{n_1+n_2+1}$.

          {\ }

Now let us mention the version for currents
and focus on stationary Lagrangian integral currents mod 2 in complex projective spaces 
%(say with multiplicity one) 
with compact support,  connected regular part and  no boundary.

                    \begin{thm}[Delaunay construction 2]\label{DC2}
       Let $T_1'$  
                   and
                          $T_2'$ 
                          be two stationary Lagrangian integral currents mod 2 in $\mathbb CP^{n_1}$ and $\mathbb CP^{n_2}$ ($n_1+n_2>0$) as above.
                          Then, based on them, infinitely many stationary Lagrangian currents mod 2 can be constructed  in $\mathbb CP^{n_1+n_2+1}$. 
                   \end{thm}

                     \begin{proof}
                   Note that Corollary \ref{corcurrent} is valid for stationary Lagrangian integral currents mod 2 
                   with compact support,  connected regular part and  no boundary.
                  We can still have global horizontal lifts 
                   (stationary Legendrian multiplicity one integral current)
                          $T_1$ and $T_2$ in $\mathbb S^{2n_1+1}$ and $\mathbb S^{2n_2+1}$ respectively.
                   Due to the connectedness, 
                   %following the patching, 
                   both regular parts of $T_1$ and $T_2$ are connected and orientable. 
                   %Hence $\ell_{1}$ and $\ell_2$ are integer divisors of $n_{1}+1$ and $n_2+1$ respectively.
                  %The Hopf projection $\pi$ maps supp$(T_j)$ onto supp$(T_j')$ as an $\ell_j:1$ mapping a.e.
                   Similarly as argued  in the above, 
                       every solution curve $\gamma$ with $C_1=-1$ and $J_{1}(C_2)\in \pi \mathbb Q$
                   can induce a stationary Legendrian current by the image of $G_\gamma(T_1, T_2)$ 
                    (where $G_\gamma$ regarded as a generating action).
                So can 
                         the image of  $\pi\circ G_\gamma$ 
                          for a stationary Lagrangian integral current mod 2 in $\mathbb CP^{n_1+n_2+1}$.
                   \end{proof}

      % \begin{rem}
      % Note that an unorientable piece 
       %\end{rem}

{\ }

 \section{Special Lagrangian cones}\label{P4}

(a) Based on Lemma \ref{key}, every connected embedded closed minimal Lagrangian submanifold will have an embedded closed special Legendrian submanifold as global horizontal lift.
%the cone over which is a regular special Lagrangian cone.
%
By applying our spiral minimal products for two embedded closed special Legendrian submanifolds (of dimension $n_1, n_2$ satisfying $n_1+n_2>0$) with $C_1=-1$ and $ J_{1}(C_2)\in \pi\mathbb Q$,
we get infinitely many embedded closed special Legendrian submanifolds in $\mathbb S^{2n_1+2n_2+1}$,
hence regular special Lagrangian cones in $\C^{n_1+n_2+1}$.

(b) The work \cite{CM} establishes the result that for every positive integer $N$ there exist an $N$-dimensional family of minimal Lagrangian tori in $\mathbb CP^2$  
 and hence an $N$-dimensional family of special  Legendrian tori in $\mathbb S^5$.
        Let $M_1$ run all these uncountably many choices of special  Legendrian tori and $M_2$ be a global horizontal lift of some connected embedded closed minimal Lagrangian submanifold.
        Then the spiral minimal products with $C_1=-1$ and $ J_{1}(C_2)\in \pi\mathbb Q$
        lead to  uncountably many special  Legendrian submanifolds, the cones over which form uncountably many regular special Lagrangian cones.
        
  (c)      Similarly, the moduli space of minimal Lagrangian immersions of connected closed submanfiolds in complex projective spaces 
        can be ``embedded" into the moduli space of  special Legendrian immersions of connected closed submanfiolds in complex projective spaces in odd dimensional spheres.
        
     (d)  In the realm of geometric measure theory, 
        the framework can start from 
                   stationary Lagrangian integral currents mod 2 in complex projective spaces
                   with compact support,  connected regular part and  no boundary.
                   Note that each of their horizontal lifts 
                   automatically has compact support,  orientable connected regular part and  no boundary.
                   Hence each induces a stationary Legendrian current, the cone of which is a special Lagrangian cone.
         
         (e) If one uses a solution curve $\gamma$ with $C_1=-1$ and $ J_{1}(C_2)\notin \pi\mathbb Q$ to replace that  in (a),
              then based on  any pair of connected embedded closed minimal Legendrian submanifolds
              their spiral minimal product  $G_\gamma$ 
              is a connected immersed non-compact minimal Legendrian submanifold without self-intersection
              (again by a calibration argument or the Almgren big regularity theorem).
              It can be observed that the cone over it is a ``regular" special Lagrangian cone with infinite density everywhere in its support.
              This reveals that 
                   regular special Lagrangian cones (assigned with finite multiplicity) are relatively rare in the family in the sense of $C_2$.
                   Similar phenomena exist as well for the categories of (c) and (d).
                   
                  % {\ }
                   
                   Although
                    with $C_1=-1$ and a convergent sequence $\{C_2\}$
 the  local solution curves $\{\gamma^0\}$   converge to a limit local solution curve,
          the ``complete" solution curves behave dramatically differently in large scale.
          For any $C_2'<C_2''$ with $J_1(C_2')\neq J_1(C_2'')$, 
          there exists $C_2'<C_2<C_2''$  such that $J_1(C_2)\notin \pi\mathbb Q$ which induces a special Legendrian current of infinite mass
          separating those corresponding to $C_2'$ and $C_2''$ (in particular for those with $J_1(C_2'),\, J_1(C_2'')\in \pi\mathbb Q$). 
          There might be some deeper mysterious reason behind.
          
{\ }

%
%
%
                       %
                       %
                       %

%{\ }

\end{document}